\documentclass[11pt]{article}

\usepackage{jones}
\usepackage{mathrsfs}
\hypersetup{pdftitle={Regularity lemma for Gaussian polynomials}}

\DeclareMathOperator{\gr}{gr}
\let\gr\gr
\newcommand{\uH}{\mathscr H}
\newcommand{\uD}{\mathscr D}
\newcommand{\uP}{\mathscr P}
\newcommand{\U}{\mathcal U}

\newcommand{\toU}{\underset{n \to \cU}{\longrightarrow}}

\newcommand{\grH}{\mathcal{H}^{\textnormal{gr}}}

\title{Regularity lemma for Gaussian polynomials}
\author{Chris Jones\thanks{UC Davis, Department of Mathematics. \texttt{chijones@ucdavis.edu}}}
\date{\today}

\begin{document}

\maketitle
\vspace{-0.3in}
\begin{abstract}
We prove a regularity lemma for polynomials with fixed degree $d$ and any number of standard Gaussian input variables.
Following modern approaches to the Szemer{\'e}di regularity lemma,
we deduce our results from a general formulation of the regularity lemma using ultraproducts.
Along the way, we give a direct combinatorial proof of the celebrated ``fourth moment theorem''
of Nualart and Peccati, which characterizes the asymptotically Gaussian
polynomials inside degree-$d$ Gaussian chaos as those such that $\E p^2 \to 1$ and $\E p^4 \to 3$.
We apply the regularity lemma to obtain a short proof of the robust Borell inequality of Mossel and Neeman.
\end{abstract}

\section{Introduction}

Let $X_1, \dots, X_n \overset{\text{i.i.d.}}{\sim} \calN(0,1)$.
How many ``different'' degree-$d$ polynomials $p(X_1, \dots, X_n)$ are there, in the sense of, say, the law of $p$ as a random variable?
While $p$ seems like it could become arbitrarily complicated as $n \to \infty$,
this is not the case. Some examples are:
\begin{enumerate}[1.]
    \item $p(X_1) = \He_d(X_1)$ where $\He_d$ is the degree-$d$ probabilists' Hermite polynomial.
    \item $p(X_1, \dots, X_n) = \frac{1}{\sqrt{n d!}}\sum_{i=1}^n \He_d(X_i)$. By the central limit theorem, the law of this random variable is asymptotically $\calN(0,1)$ as $n \to \infty$.
    \item $p(X_1, \dots, X_n) = \frac 1n \sum_{1 \leq i < j \leq n}X_iX_j$. This is asymptotic to $\frac 12 \He_2(Z)$ for $Z \sim \calN(0,1)$.
\end{enumerate}

Despite the limit $n \to \infty$ in examples 2 and 3,
the law of $p$ approaches the law of a polynomial in just one standard Gaussian variable.
Herry et al.\ \cite{herry2024limit} proved that this is a general phenomenon:
any sequence of degree-$d$ Gaussian polynomials which converges in law, converges in law towards another Gaussian polynomial (possibly of lower degree).

Here we characterize the asymptotic limits
of Gaussian polynomials in much stronger ways,
broadly termed ``regularity lemmas for Gaussian polynomials''.
First, we prove a regularity lemma in the traditional sense of
approximation by a structured constant-size object.
For a set of variables $X$,
let $\mathcal{H}_d(X)$ be the span of the degree-$d$ Hermite polynomials in $X$, and let $\mathcal{H}_{\le d}(X)$ be the span of all monomials in $X$
up to degree $d$. $\cH_d(X)$ is also called the ``degree-$d$ Gaussian chaos'' for $X$.

\begin{theorem}[Regularity lemma, approximation version]\label{thm:approximation-intro}
    For every $d \in \N$, $\eps > 0$ there exists $M = M(d,\eps)$ such that:
    for all $p \in \mathcal{H}_{\le d}(X)$ such that $\E p^2 \leq 1$,
    there is $m \leq M$, a polynomial $q(Z_1, \dots, Z_m), \deg(q) \leq d$ and $Z_a \in \calH_{\le d}(X)$ such that:
    \begin{enumerate}[(i)]
        \item $\E (p-q)^2 \leq \eps$
        \item $W_2((Z_1, \dots, Z_{m}), \calN(0,1)^{\otimes m}) \leq \eps$ where $W_2$ is 2-Wasserstein distance.
    \end{enumerate}
\end{theorem}

To prove \cref{thm:approximation-intro},
we will deduce it from a semantically stronger theorem
characterizing the \emph{ultraproduct} of Gaussian polynomials~\cite{elek2012measure,tao2013ultraproducts}.
Roughly speaking, our theorem states that the possible limits of Gaussian polynomials
are ``graded Gaussian polynomials''.

At an intermediate level between the completely finitary \cref{thm:approximation-intro}
and the ultraproduct viewpoint,
one can also view a regularity lemma as a ``compactification''.
For example, the space of graphons compactifies
the space of graphs under the cut metric~\cite{lovasz2006limits,elek2012measure,lovasz2012large}.
At this level, we prove that graded-degree-$d$ Gaussian polynomials are the closure of degree-$d$ Gaussian polynomials in distribution (and the compactification when the space is bounded, $\E p^2 \leq 1$).

\begin{theorem}[Regularity lemma, compactification version]\label{thm:compactification-intro}
Let $X$ be a set of countably many variables,
and let $Z$ be a set of graded variables with countably many variables of each grade.
Let $\grH_d(Z)$ be the span of graded-degree-$d$ Hermite polynomials in $Z$.
Let $d \in \N$.
\begin{enumerate}[(i)]
    \item For all $p \in \grH_d(Z)$, there exists a sequence $p_n \in \cH_d(X_1, \dots, X_n)$ such that $p_n \xrightarrow{d} p$.
    \item If $p_n \in \grH_d(Z)$ satisfies $p_n \xrightarrow{d} p$ for some $p$, then $p_n \xrightarrow{d} p'$ for some $p' \in \grH_d(Z)$.
\end{enumerate}
\end{theorem}

Equivalently, $\{\Law(p): p \in \grH_d(Z)\}$ is the closure of $\{\Law(p) : p \in \cH_d(X)\}$ in the space of probability measures on $\R$ with the weak topology.
A consequence of (ii) is that $\{\Law(p): p \in \grH_d(Z), \E p^2 \leq 1\}$ is compact in the weak topology.

In the main text, we will prove multivariate versions of these theorems.
Using a hypercontractivity argument, we can also generically upgrade \cref{thm:compactification-intro} from the weak topology to Wasserstein metrics, see \cref{sec:hypercontractivity}.
Considering it to be a problem of independent interest, Herry et al.\ \cite{herry2024limit} proved a variant of this theorem, which is that $\cH_{\le d}(X)$ is closed in distribution when $X$ is a set of countably many standard Gaussian variables.

There is a large body of work on characterizing the Gaussian limits inside $\calH_d(X)$.
The famous ``fourth moment theorem'' of Nualart and Peccati~\cite{nualart2005central,nualart2008central,nourdin2015optimal}
states that, if $\E p_n^2 \to 1$, then $p_n \xrightarrow{d} \calN(0,1)$ if and only if $\E p_n^4 \rightarrow 3$
if and only if all ``contraction norms'' of $p_n$ vanish.
We identify a small extension: if and only if
all proper flattenings of the coefficient tensor
of $p_n$ have vanishing spectral norm.
We use this to give a direct combinatorial proof of the fourth moment theorem
and then to prove the regularity lemma.

Our idea for the regularity lemma was motivated by an application to the theory of the Grothendieck constant \cite{heilman2026primal}.
In \cref{sec:borell} we give another application:
a much shorter proof of the ``robust Borell inequality'' of Mossel and Neeman \cite{mossel2015robust}.
Borell's inequality states that the maximum possible noise stability of a set $A \subseteq \R^n$ with fixed Gaussian measure is a halfspace.
The robust version states that any near-maximizer is close to a halfspace in $L_2$.
The regularity lemma gives robustness results like this almost for free, by a general strategy:
\begin{enumerate}[1.]
    \item Suppose we want to determine which Gaussian polynomial $p \in \cH_{\le d}(X)$ maximizes a target quantity.
    As long as the target is suitably continuous (such as noise stability), there exists a maximizer in the compactified space of graded Gaussian polynomials.
    \item Identify the maximizers on the compactified space. Perturbative arguments
    can be used here (we can search for local optima, then select the largest, to identify the global optimum).
    \item Robustness is free: by compactness and continuity,
    for every $\eps > 0$, there exists $\delta > 0$
    such that all $\delta$-near-optimal points are within $\eps$ of global optima (in a suitable metric).
\end{enumerate}
For the robust Borell inequality, we use a slight variant of this argument mainly because set indicator functions are not polynomial functions.
Since noise stability smooths away high-degree components, approximation by a low-degree polynomial still works.

\paragraph{Statement of AI use.} An initial draft of the regularity lemma was written in April 2026 using some assistance from GPT-5.4 (given the tensor decomposition strategy, it was able to suggest using the fourth moment theorem, although I could not understand its proof strategy).
For the current version, GPT-6 Astra helped close technical gaps and extend the proofs.
While the first version used an energy-increment argument,
GPT-6 Astra suggested the ultraproduct approach (communicated by Giulio Malavolta).
The author takes responsibility for the contents of the paper and any remaining errors.

\section{Graded Gaussian polynomials}
\label{sec:graded}

A graded variable is a variable with an integer label $r>0$ called its ``grade''.
We write $X_1, X_2, X_3,\dots$ for ``original'' grade-one
variables and $Z^{(r)}_1, Z_2^{(r)}, Z_3^{(r)},\dots$ for grade-$r$ variables,
with the identification $X_i = Z_i^{(1)}$ at grade one.
Sample $Z_i^{(r)} \overset{\text{i.i.d.}}{\sim} \calN(0,1)$ independently across all variables regardless of grade.

The idea of the asymptotic limit is that an asymptotically Gaussian
random variable in $\cH_d(X)$ can be replaced in the limit by a single new grade-$d$ variable.
In this way, we will represent the limit of a sequence of Gaussian polynomials
by a single graded polynomial. When we do an operation on this limit object
which depends on the degree (e.g.\ calculate the degree-$d$ weight)
then we use the grade information to recover the ``original'' degree.

We use the following standard notions from Gaussian analysis.
Let $\He_k$ be the unnormalized probabilists' Hermite polynomials.
Let $\psi_k = \He_k/\sqrt{k!}$ be their normalizations.
For a finitely-supported multi-index $\al \in \N^X$
let $X^\alpha = \prod_i X_i^{\alpha_i}$ and similarly for $\He_\alpha$ and $\psi_\alpha$. Let
$|\alpha|=\sum_i\alpha_i$ and $\alpha! = \prod_i \alpha_i!$.
Then $\{\psi_\alpha(X) : \alpha \in \N^X, |\alpha|<\infty\}$ is an orthonormal basis of
$L_2(X)$.
Let $\cH_d(X)$ be the closed span of the degree-$d$ Hermite polynomials in $X$.
Let $\cH_{\le d}(X) = \oplus_{i = 0}^d \cH_{i}(X)$. Let $\ip{f,g} = \E fg$ and $\|f\|_2 = (\E f^2)^{1/2}$.

Polynomials in $Z$ have an additional notion of total grade, or graded degree.
Let $\gr(\alpha) = \sum_{i,r} r\alpha_{ir}$ be the total grade of a multi-index $\al \in \N^Z$.
The Hermite polynomial $\He_\alpha(Z)$ is defined as before, and $\psi_\alpha(Z)$ are still an orthonormal basis for $L_2(Z)$.
Define graded-degree-$d$ chaos by
\begin{equation}\label{eq:limit-spaces}
    \grH_d(Z) = \overline{\linspan}\left(\He_\al(Z) : \gr(\al) = d\right)\,.
\end{equation}
The notation $\overline{\linspan}$ denotes the closure of the linear span in $L_2(Z)$.

The following notions on Gaussian space interact with degree, and so behave differently on graded Gaussian space.
Let $\Pi_d$ be the operator on $L_2(Z)$ which projects to $\grH_d(Z)$.
Let $W_d(f) = \norm{\Pi_d f}_2^2$.
For variables with grade $r$, the Ornstein--Uhlenbeck noise operator at parameter $t \in [0,1]$ acts by sampling standard Gaussians $W$ with covariance $\E[W_i^{(r)} Z_i^{(r)}] = t^r$.
Under these definitions, we still have the noise correlation and Parseval identities,
\begin{equation}\label{eq:correlation}
 \E_{(W,Z) \sim t\text{-correlated}}[f(W)g(Z)]=\sum_{d\ge0}t^d\langle\Pi_d f,\Pi_d g\rangle\,, \qquad \qquad \E[f(Z)^2] = \sum_{d \geq 0} W_d(f)\,.
\end{equation}

We verify that all $p \in \grH_d(Z)$ can be approximated by finite Gaussian polynomials, which is the first (easier) part of \cref{thm:compactification-intro}.
\begin{lemma}\label{lem:approximability}
    Let $d_k \in \N$ for $k \in \N$.
    For all $p_k \in \grH_{d_k}(Z)$
    there exist sequences $p_{n,k} \in \cH_{d_k}(X_1, \dots, X_{n})$ such that $(p_{n,k})_{k \in \N} \xrightarrow{d} (p_k)_{k \in \N}$.
\end{lemma}
\begin{proof}
For each graded variable $Z^{(r)}_i$ introduce new standard Gaussians $X_{ir,j}$ for $j \in [N]$.
Define approximants for $Z^{(r)}_i$ and $\He_m(Z^{(r)}_i)$ by,
\[
    S_{ir} = \frac{1}{\sqrt{N}} \sum_{j=1}^N \psi_r(X_{ir,j}), \qquad U_{m,ir} = \frac{1}{N^{m/2}} \sum_{\substack{j_1, \dots, j_m \in [N]\\ j_1, \dots, j_m \text{ distinct}}} \psi_r(X_{ir,j_1})\psi_r(X_{ir,j_2})\cdots \psi_r(X_{ir,j_m})\,.
\]
We approximate $p_k \in \grH_{d_k}(Z)$ by:
\[
    p_{N,k} = \sum_\alpha c_\alpha \prod_{i,r} U_{\alpha_{ir},ir} \qquad \text{where}\qquad p_k = \sum_\alpha c_\alpha \He_\alpha(Z) = \sum_\alpha c_\alpha \prod_{i,r} \He_{\alpha_{ir}}(Z_i^{(r)})\,.
\]
All terms in $p_{N,k}$ have Hermite degree $d_k$ so $p_{N,k} \in \calH_{d_k}(X)$. It remains to show $(p_{N,k})_{k \in \N} \xrightarrow{d} (p_k)_{k \in \N}$ as $N \to \infty$.
This is equivalent to $(p_{N,1}, \dots, p_{N,k}) \xrightarrow{d} (p_1,\dots, p_k)$ for every finite $k$.

First assume each $p_k=\He_{\alpha^k}(Z)$ is a single Hermite polynomial.
Consider $\tilde p_{N,k} = \He_{\alpha^k}(S)$. By the multivariate central limit theorem, $(\tilde p_{N,1}, \dots, \tilde p_{N,k}) \xrightarrow{d} (p_1, \dots, p_k)$.
On the other hand,
\begin{equation}\label{eq:multilinear-apx}
    \E (\tilde p_{N,k} - p_{N,k})^2 = O(N^{-1})\,.
\end{equation}
This is a combinatorial moment calculation which we include in \cref{lem:omitted-moment}.

Combine \cref{eq:multilinear-apx} and the convergence $(\tilde p_{N,1}, \dots, \tilde p_{N,k}) \xrightarrow{d} (p_1, \dots, p_k)$ with Slutsky's lemma to deduce $( p_{N,1},\dots, p_{N,k}) \xrightarrow{d} (p_1, \dots, p_k)$.
This completes the case of single Hermite polynomials.

The continuous mapping theorem gives $(p_{N,1}, \dots, p_{N,k}) \xrightarrow{d} (p_1, \dots, p_k)$ if every
$p_k$ is a finite sum.
For $p_k\in\grH_{d_k}(Z)$ with an infinite number of terms,
take a sequence of finite $L_2$ approximants, approximate each
one, then take a diagonal sequence to obtain $(p_{N,1},\dots,p_{N,k}) \xrightarrow{d} (p_1,\dots, p_k)$.
\end{proof}

\section{Proof of the regularity lemma}
\label{sec:regularity}

\subsection{Regularity lemma for degree 1 and 2}

A degree-1 polynomial is of the form $p(X) = \sum_i c_i X_i$.
The law of this polynomial is always Gaussian, so the characterization of degree-1 Gaussian polynomials is immediate.

As a warm-up, we prove a version of \cref{thm:approximation-intro} for degree-2 Gaussian polynomials.
\begin{lemma}\label{lem:degree2}
    Let $\eps > 0$.
    For all $p \in \cH_2(X_1, \dots, X_n)$ such that $\E p^2 \leq 1$, there exist $Z_0 \in \cH_2(X_1, \dots, X_n)$ and $Z_1, \dots, Z_m \in \cH_1(X_1, \dots, X_n)$ with $m \leq 2\eps^{-2}$ and coefficients $\lam_a$ such that:
    \begin{enumerate}[(i)]
        \item $p = Z_0 + \sum_{a=1}^m \lam_a\He_2(Z_a)$.
        \item $Z_0, Z_1,\dots, Z_m$ are independent and $W_2(Z_0, \calN(0, s)) \leq C\eps$ for an absolute constant $C$ where $s = \E Z_0^2$.
    \end{enumerate}
\end{lemma}

\begin{proof}
Write $p$ in the Hermite basis as $p(X) = \sum_{i \neq j} P_{ij} X_iX_j + \sum_{i=1}^n P_{ii} \He_2(X_i)$ for a symmetric coefficient matrix $P \in \R^{n \times n}_\sym$.
Let $\sum_{a}\lam_a v_av_a^\T$ be the spectral decomposition of $P$.
Then
\begin{equation}
     p(X) = X^\T P X - \tr(P) = \sum_{a} \lam_a \ip{v_a, X}^2 - \tr(P) = \sum_{a} \lam_a (\ip{v_a, X}^2 - 1) = \sum_a \lam_a \He_2(\ip{v_a, X})\,.
\end{equation}
By the orthogonality of the eigendecomposition, $\ip{v_a, X}$ are independent standard Gaussians.

We have $\frac 12 \E p^2 = \sum_{i,j=1}^n P_{ij}^2 =: \norm{P}_F^2$.
On the other hand, $\norm{P}_F^2 = \sum_a \lam_a^2$.
Since $\E p^2 \leq 1$, we deduce that $\sum_a \lam_a^2 \leq \frac 12$.

Let $A = \{a : |\lam_a| \geq \eps/2\}$. By the above, $|A| \leq 2\eps^{-2}$.
Write
\begin{align*}
    p = \sum_{a \in A} \lam_a \He_2(\underbrace{\ip{v_a, X}}_{=: Z_a})
   + \underbrace{\sum_{a\notin A}\lam_a \He_2(\ip{v_a, X})}_{=: Z_0}\,. 
\end{align*}

Since the $\He_2(\ip{v_a, X})$ are independent and mean-zero,
a central limit theorem implies $Z_0$ is approximately Gaussian if there are many summands.
If there are few summands, $Z_0$ is small since $|\lam_a| \leq \eps/2$.
Both cases are handled simultaneously by the CLT from \cite[Theorem 4.1]{rio2009upper} which yields for an absolute constant $C$, with $s = \E Z_0^2 = 2\sum_{a\not\in A} \lam_a^2$,
\[
    W_2(Z_0, \calN(0,s)) \leq \frac{C}{\sqrt{s}} \left(\sum_{a \notin A} \lam_a^4 \right)^{1/2} \leq C \max_{a \notin A}|\lam_a| \leq C \eps\,. \qedhere
\]
\end{proof}

\subsection{Preliminaries on ultraproducts}

A collection $\cU$ of subsets of $\N$ is a \emph{non-principal ultrafilter} if:
\begin{enumerate}[(i)]
    \item $\cU$ contains $\N$ but not $\0$.
    \item If $A \in \cU$ and $A \subseteq B$, then $B \in \cU$.
    \item If $A,B \in \cU$ then $A \cap B \in \cU$
    \item If $A \subseteq \N$ then exactly one of $A$ and $\N \setminus A$ is in $\cU$.
    \item There are no finite sets in $\cU$.
\end{enumerate}
We say that $A \in \cU$ is a \emph{$\cU$-large} set of natural numbers.

Fix any\footnote{
    A non-principal ultrafilter $\cU$ can be easily constructed using Zorn's lemma, which uses the axiom of choice (ZFC).
    We consider it all but certain that \cref{thm:regularity-lemma-compactness,thm:regularity-approximation} can also be proven without choice.
} non-principal ultrafilter $\cU$ on $\N$.
We can use $\cU$ to ``pick'' the limit of a sequence of objects inside a compact Hausdorff space,
as summarized in the next lemma.

\begin{lemma}[{\cite[Theorems~13.7 and~17.4]{willard1970general}}]
\hspace{0.1cm}\label{lem:ultralimits}
    \begin{enumerate}[(i)]
        \item (Existence and uniqueness of limits)
        Let $S$ be a compact Hausdorff space.
        For all sequences $x_n \in S$ there exists a unique $x \in S$ such that,
        for all open neighborhoods $V$ of $x$, we have $\{n : x_n \in V\} \in \cU$.
        Write $x = \lim_{n \to \calU} x_n$ or $x_n \underset{n \to \cU}{\longrightarrow} x$.

        \item (Continuity of limits) Let $S, T$ be Hausdorff spaces and $f: S \to T$ continuous.
        If $x_n \toU x$ in $S$ then $f(x_n) \toU f(x)$ in $T$.
        
        \item (Extraction of a subsequence)
            Let $x_n\toU x$ in a first-countable Hausdorff space $S$
            (for example, a metrizable space).
            Then there exists an infinite subsequence such that $x_{n}\to x$
            in $S$ along the subsequence.
            Moreover, the subsequence may be chosen inside
            any $A\in\cU$.
        \end{enumerate}
\end{lemma}

While \cref{lem:ultralimits} applies to compact spaces,
sequences inside a non-compact space can still be studied
using \emph{ultraproducts}.
We are specifically interested in sequences and limits of Gaussian polynomials. Define
\[
    B_d = \{(p_n)_{n \in \N}: p_n \in \cH_{\le d}(X_1, \dots, X_n),\, \sup_n \E p_n^2 < \infty\}\,, \qquad N_d = \{(p_n)_{n \in \N} \in B_d : \lim_{n \to \cU} \E p_n^2 = 0\}\,.
\]
The scalar ultralimit in the definition of $N_d$ exists because $\E p_n^2$ is bounded and so $(\E p_n^2)_{n \in \N}$ lies in a compact subset of $\R$.
By the triangle inequality, $N_d$ is a linear subspace of $B_d$.
The \emph{Hilbert ultraproduct} is the vector space
\[
    \uH_{\le d} = B_d/N_d\,.
\]

We write $[p_n]$ for elements of $\uH_{\le d}$
which formally are equivalence classes of the form $[p_n] = \{(p_n + r_n)_{n \in \N} : (r_n)_{n \in \N} \in N_d\}$.
A crucial property of $\uH_{\le d}$ is that it inherits the Hilbert space structure from $\cH_{\le d}(X)$.
\begin{lemma}\label{lem:ultra-hilbert}
    $\uH_{\le d}$ is a Hilbert space with inner product $\ip{[p_n], [q_n]}=\lim_{n \to \U}\E p_nq_n$.
\end{lemma}
\begin{proof}
It is easy to check that $\ip{[p_n], [q_n]}$ defines a valid inner product on $\uH_{\le d} = B_d/N_d$.
What remains is to check that the space is complete.

Consider the space $B_d$ with the supremum norm $\norm{(p_n)_{n \in \N}}_\infty = \sup_n \norm{p_n}_2$.
Since each coordinate space $\calH_{\le d}(X_1, \dots, X_n)$ is complete, $B_d$ is complete in this norm ($B_d$ is a Banach space).
Moreover, $(p_n)_{n \in \N} \mapsto\lim_{n\to\cU}\norm{p_n}_2$ is
$1$-Lipschitz in this norm, so the zero set $N_d$ is closed in this norm.
By \cite[Theorem 1.41(d)]{rudin1991functional}, $B_d/N_d$ is complete in the quotient norm $\norm{[p_n]}_{B_d/N_d} := \inf_{(r_n)_{n \in \N} \in N_d}\norm{(p_n+r_n)_{n \in \N}}_\infty$.

It remains to show that the quotient norm agrees with the
inner-product norm $\sqrt{\ip{[p_n], [p_n]}} = \lim_{n \to \cU} \|p_n\|_2$
which will prove that $\uH_{\le d}$ is complete in the inner-product norm.
Fix $[p_n]$ and set $L=\lim_{n\to\cU}\norm{p_n}_2$.
Every representative $(q_n)_{n \in \N} \in B_d$ of $[p_n]$ satisfies
$\lim_{n\to\cU}\norm{q_n}_2=L$. Hence
$\sup_n\norm{q_n}_2\ge L$ by \cref{lem:ultralimits}(iii) and therefore the quotient norm is at least $L$.

Conversely, for every $\varepsilon>0$, the set $A=\{n\in\N:\norm{p_n}_2<L+\varepsilon\}$
belongs to $\cU$ by \cref{lem:ultralimits}(i). Define $q_n=p_n$ for $n\in A$ and
$q_n=0$ otherwise. Then $(p_n-q_n)\in N_d$ and
$\sup_n\norm{q_n}_2\le L+\varepsilon$.
Therefore the quotient norm is at most $L+\eps$. Letting $\eps\downarrow 0$ shows that it
equals $L$, completing the proof.
\end{proof}

Let $\uH_d = \{[p_n] : \forall n\in\N. \; p_n \in \cH_d(X)\}$ be the degree-$d$ component.
Since $\|\Pi_k p\|_2 \leq \|p\|_2$ then $N_d$ vanishes under $\Pi_k$ and $\Pi_k$ is well-defined on $\uH_{\le d}$. Hence, we maintain the decomposition $\uH_{\le d} = \oplus_{j=0}^d \uH_j$.

Two remarks on the purpose of introducing the ultraproduct:
\begin{enumerate}[1.]
    \item A technical advantage of the ultralimit is that it consistently chooses a limit
    inside all compact spaces simultaneously.
    If $x_n$ are in a compact Hausdorff space (e.g.\ bounded real numbers)
    then we can treat $\lim_{n \to \cU} x_n$ like a normal limit
    and delay applying compactness, passing to subsequences,
    etc.\ until we are ready to apply \cref{lem:ultralimits}(iii) to ``exit'' the ultraproduct.
    \item 
    For sequences in a non-compact space,
    ontologically, we can think of $[p_n] \in \uH_{\le d}$
    as a new type of object representing the limit of the sequence $p_n$.
    These objects inherit all of the usual operations of the original space, for example,
    given $[p_n], [q_n] \in \uH_{\le d}$ we can select representatives $p_n, q_n$
    and form $[p_n + q_n] \in \uH_{\le d}$.
    Classifying the limits of Gaussian polynomials reduces to classifying
    the elements of $\uH_{\le d}$
    (starting with \cref{lem:ultra-hilbert} which shows they form a Hilbert space).

    In the asymptotic Hilbert space, asymptotic approximations become exact relations. By \cref{lem:ultra-hilbert},
    \begin{align*}
        \|x_n - y_n\|_2 \toU 0 &\text{ if and only if } [x_n] = [y_n] \text{ in }\uH_{\le d}\\
        \ip{x_n, y_n} \toU 0 &\text{ if and only if } [x_n] \perp [y_n]\text{ in }\uH_{\le d}\,.
    \end{align*}
\end{enumerate}

\subsection{Orthogonal decomposition of the asymptotic Hilbert space}

Since $\uH_d$ is a Hilbert space, by Zorn's lemma it has an orthonormal basis, which is uncountable.
The goal of this section is to construct an orthonormal basis of ``graded Hermite polynomials''
for $\uH_d$ which will solve, in a way, the problem of describing $\uH_d$.

For \(p\in\cH_r(X)\) and \(q\in\cH_s(X)\), define the
\emph{top product} (or \emph{Wick product})
\[
 p\diamond q=\Pi_{r+s}(pq).
\]
It can be seen that $\diamond$ is commutative and associative.
We can also lift the top product to $\uH_d$.

\begin{lemma}
    $\diamond$ is a well-defined and continuous bilinear map $\uH_r \times \uH_s \to \uH_{r+s}$.
\end{lemma}
\begin{proof}
By contractivity of projection, Cauchy--Schwarz,
and hypercontractivity (\cref{lem:hypercontractivity}),
\[
    \norm{p\diamond q}_2
    \le \norm{pq}_2
    \le \norm{p}_4\norm{q}_4
    \le 3^{(r+s)/2}\norm{p}_2\norm{q}_2.
\]
Consequently, if $p_n \in \cH_r(X)$ and $q_n \in \cH_s(X)$ satisfy $\sup_n \E p_n^2 < \infty$ and $\sup_n \E q_n^2 < \infty$,
then $\sup_n \E (p_n\diamond q_n)^2 < \infty$ and hence $(p_n\diamond q_n)_{n \in \N} \in B_{r+s}$. Lift $\diamond$ to $\uH_{r} \times \uH_s$ by $[p_n]\diamond[q_n]:=[p_n\diamond q_n]$.
To check that $\diamond$ is independent of the representatives, suppose
$[p_n]=[p'_n]$ and $[q_n]=[q'_n]$.
By bilinearity and the preceding bound,
\[
\begin{split}
    \norm{p_n\diamond q_n-p'_n\diamond q'_n}_2
    \le 3^{(r+s)/2}\bigl(
        \norm{p_n-p'_n}_2\norm{q_n}_2
        +\norm{p'_n}_2\norm{q_n-q'_n}_2
    \bigr).
\end{split}
\]
The right-hand side tends to zero along $\cU$, since $\|q_n\|_2, \|p'_n\|_2$
are bounded and the difference terms tend to zero along $\cU$. Thus the product is well-defined.

Finally, taking ultralimits in the first inequality gives
\[
    \norm{u\diamond v}_{\uH_{r+s}}
    \le 3^{(r+s)/2}
       \norm{u}_{\uH_r}\norm{v}_{\uH_s}.
\]
Hence $\diamond$ is a bounded bilinear map and therefore
jointly continuous.
\end{proof}

Let $\uD_d$ be the subspace of $\uH_d$ generated by lower-degree polynomials,
\begin{equation}\label{eq:decomposition}
 \uD_d=\overline{\linspan}
   \{u\diamond v:u\in\uH_r,\ v\in\uH_{d-r},\ 0< r<d\}\,.
\end{equation}
Since $\uH_d$ is a Hilbert space, we can form the orthogonal decomposition $\uH_d = \uD_d\oplus\uP_d$ where $\uP_d$ is the orthogonal complement of $\uD_d$.
We call the elements of \(\uP_d\) \emph{primitive}.

Select an arbitrary orthonormal basis $\{[e_{n,i}] : i \in I_d\}$ for $\uP_d$
and let $I = \sqcup_{d \in \N} I_d$. We also normalize $\E e_{n,i}^2 = 1$ exactly.
For finitely-supported $\alpha \in \N^{I}$ define $\gr(\alpha) = \sum_{d \in \N}\sum_{i \in I_d}d \alpha_i$.
By the definition of $\uD_d$ and $\uP_d$ there is a dense spanning set for $\uH_d$ given by
\[
    \Big\{\frac{1}{\sqrt{\alpha!}}\underset{i \in I}{\diamond} [e_{n,i}^{\diamond \alpha_i}] : \al \in \N^I, \gr(\alpha) = d\Big\}\,.
\]

We will prove that all primitive elements are asymptotically Gaussian,
and these top products are asymptotically orthonormal Hermite polynomials with respect to the primitive Gaussians.
This will characterize the basis for $\uH_d$ as the collection of graded Hermite polynomials
in the primitive elements.

\begin{theorem}[Regularity lemma, ultrapower version]
\label{thm:regularity-ultralimit}
$\{[\psi_\alpha(e_n)] : \alpha \in \N^I, \gr(\alpha) = d\}$ is an orthonormal basis for $\uH_{d}$.
Furthermore, let $(Z_i)_{i \in I}$ be independent standard Gaussian variables such that $I_d$ is assigned grade $d$.
Let $d_k \in \N$ for $k \in \N$. For any $[p_{n,k}] \in \uH_{d_k}$ write the basis representation of $[p_{n,k}]$ and define $p_k \in \grH_{d_k}(Z)$ by
\[
    [p_{n,k}] = \sum_{\alpha \in \N^I : \gr(\alpha) = d_k} c_{k,\alpha} [\psi_\alpha(e_n)]\,, \qquad \qquad p_k = \sum_{\alpha \in \N^I: \gr(\alpha) = d_k} c_{k,\alpha} \psi_\alpha(Z)\,.
\]
Then $(p_{n,k})_{k \in \N} \overset{d}{\toU} (p_k)_{k \in \N}$.
\end{theorem}

Note that the target $(p_k)_{k \in \N}$ of the convergence in distribution
is independent of the choice of representatives $p_{n,k}$ for $[p_{n,k}]$
by Slutsky's lemma applied to the difference element from $N_d$.

\subsubsection{Proof of \texorpdfstring{\cref{thm:regularity-ultralimit}}{orthogonal decomposition}}
\label{sec:tensors}

We use combinatorial arguments on the tensor representations of the polynomials.
For an ordered tuple $\mathbf i=(i_1,\ldots,i_d)$, let
$\alpha(\mathbf i)_j=|\{s:i_s=j\}|$ and set $\He_{i_1\cdots i_d}(X)
  :=\He_{\alpha(\mathbf i)}(X)$.
For $p \in \calH_d(X_1, \dots, X_n)$,
the \emph{Hermite coefficient tensor} $P \in (\R^n)^{\otimes d}_{\sym}$ is the symmetric order-$d$ tensor with dimension $n$ such that,
\begin{equation}\label{eq:coefficient-tensor}
    p(X) = \frac{1}{\sqrt{d!}}\sum_{i_1, \dots, i_d = 1}^n P_{i_1\cdots i_d} \He_{i_1\cdots i_d}(X)\,.
\end{equation}
In Nualart and Peccati's notation, $p$ is realized as a $d$-dimensional multiple Wiener--\Ito integral, $p = \frac{I_d(P)}{\sqrt{d!}}$ \cite{nualart2005central}.
The entry normalization creates an important isometry between $\cH_d(X_1, \dots, X_n)$ and $(\R^n)^{\otimes d}_\sym$:
\begin{equation}\label{eq:isometry}
    \E pq = \ip{P, Q}_F =: \sum_{i_1, \dots, i_d} P_{i_1\cdots i_d}Q_{i_1\cdots i_d}, \qquad \E p^2 = \norm{P}_{F}^2 := \sum_{i_1,\dots, i_d} P_{i_1\cdots i_d}^2\,.
\end{equation}
We call $P \in (\R^n)^{\otimes d}$ a non-symmetric Hermite coefficient tensor
for $p$ if it satisfies \cref{eq:coefficient-tensor}.
If at least one tensor $P,Q$ is symmetric in \eqref{eq:isometry}, the equality continues to hold since the other tensor can be symmetrized without loss of generality.

For $u \in \cH_r(X)$ and $v \in \cH_s(X)$ with Hermite coefficient tensors $U,V$
we observe that there is a non-symmetric Hermite coefficient tensor $W$ for $u \diamond v \in \cH_{r+s}(X)$ given by
\begin{equation}\label{eq:top-product-tensor}
    W = c_{r,s} U \otimes V, \qquad c_{r,s} = \sqrt{\frac{(r+s)!}{r!s!}}\,.
\end{equation}
The factor $c_{r,s}$ accounts for normalization in \eqref{eq:coefficient-tensor}.

Given a tensor $P$ and a partition of the indices $[d]$ into two ordered sets $(L,R)$, the \emph{flattening} $P^{(L,R)} \in \R^{n^{|L|} \times n^{|R|}}$ is the matrix with entries $P^{(L,R)}_{i_L i_R} = P_{i_1\cdots i_d}$ where $i_L = (i_\ell)_{\ell \in L}$ and $i_R = (i_r)_{r \in R}$.
An order-$d$ symmetric tensor has only $d+1$ distinct flattenings.
We write $P^{(k)} \in \R^{n^k \times n^{d-k}}$ for the $k$th flattening.
vec denotes flattening the tensor into a column vector.

\begin{lemma}\label{lem:primitive}
Let $[p_n] \in \uH_d$. Let $P_n$ be the Hermite coefficient tensor of $p_n$. Then
\[
 [p_n]\in\uP_d
 \quad\Longleftrightarrow\quad
 \lim_{n \to\U}\norm{P_n^{(r)}}=0
 \quad\text{for every }0 < r<d.
\]
\end{lemma}

\begin{proof}
Let $[p_n] \in \uH_d(X),[u_n]\in\uH_r(X),[v_n]\in\uH_{d-r}(X)$ have Hermite coefficient tensors $P_n,U_n,V_n$.
By the tensor isometry \eqref{eq:isometry} and \eqref{eq:top-product-tensor},
\begin{align*}\label{eq:diamond-isometry}
    \E p_n (u_n \diamond v_n) &= c_{r,d-r} \langle P_n, U_n \otimes V_n\rangle = c_{r,d-r} \text{vec}(U_n)^\T P_n^{(r)} \text{vec}(V_n)\\
    |\E p_n (u_n \diamond v_n)| &\leq c_{r,d-r}\norm{P_n^{(r)}} \norm{U_n}_F \norm{V_n}_F = c_{r,d-r}\norm{P_n^{(r)}}\norm{u_n}_2\norm{v_n}_2\,.
\end{align*}

Suppose that $\lim_{n \to \cU} \norm{P_n^{(r)}} = 0$.
Taking ultralimits and using that $\|u_n\|_2, \|v_n\|_2$ are bounded,
then $\ip{[p_n], [u_n \diamond v_n]}_{\uH_d} = \lim_{n \to \cU} \E p_n(u_n \diamond v_n) = 0$.
Therefore, $[p_n] \perp \uD_d$ in this case. 

Conversely, suppose \([p_n]\in\uP_d\) and $\lim_{n \to \cU} \|P_n^{(r)}\| > 0$.
Choose unit-norm singular vectors $\tilde U_n, \tilde V_n$ for each $n$ such that
\[
 (\tilde U_n)^{\T} P_n^{(r)} \tilde V_n = \norm{P_n^{(r)}}\,.
\]
Unflatten $\tilde U_n, \tilde V_n$ into tensors $U_n, V_n$.
Note that since $P_n$ is a symmetric tensor, its nonzero singular vectors also correspond to symmetric tensors.
Then $U_n, V_n$ represent unit-norm \(u_n\in\cH_r(X)\), \(v_n\in\cH_{d-r}(X)\) with,
\[
    \E p_n (u_n \diamond v_n) = c_{r,d-r} \langle P_n, U_n \otimes V_n\rangle = c_{r,d-r}(\tilde U_n)^{\T} P_n^{(r)} \tilde V_n = c_{r,d-r}\norm{P_n^{(r)}}\,.
\]
Taking the ultralimit shows $\ip{[p_n],[u_n]\diamond[v_n]}>0$ contrary to \([p_n]\perp\uD_d\).
\end{proof}

\begin{lemma}
\label{lem:wick-hermite}
For all finitely-supported $\al \in \N^I$,
\[
    \frac{1}{\sqrt{\alpha!}}\underset{i \in I}{\diamond} [e_{n,i}^{\diamond \alpha_i}] = [\psi_\alpha(e_n)]\,.
\]
\end{lemma}
\begin{proof}
Let $h_{n,\alpha} = \frac{1}{\sqrt{\alpha!}}\underset{i \in I}{\diamond} e_{n,i}^{\diamond \alpha_i}$.
We have,
\begin{align}\label{eq:var1}
    &\quad \E(\psi_\alpha(e_n) - h_{n,\alpha})^2 = \E \psi_\alpha(e_n)^2 - 2\E \psi_\alpha(e_n) h_{n,\alpha} + \E h_{n,\alpha}^2 = \E \psi_\alpha(e_n)^2 - \E h_{n,\alpha}^2\,.
\end{align}
The last equality holds because $h_{n,\alpha} \in \cH_{\gr(\alpha)}(X)$ exactly,
so it is orthogonal to lower-degree terms in $\psi_\alpha(e_n)$,
and the unique leading monomial of $\psi_\alpha(e_n)$ is the same as that of $h_{n,\alpha}$.

By \cref{lem:primitive}, the norm of every proper flattening of every $e_{n,i}$ vanishes.
By the fourth moment theorem (\cref{thm:fourth-moment}), the $e_{n,i}$ supported in $\alpha$ are asymptotically jointly Gaussian (note that the proof of \cref{thm:fourth-moment} still holds when all convergence is along $\cU$).
Then we have the convergence of the variance $\E \psi_\alpha(e_n)^2 \underset{n \to \cU}{\longrightarrow} 1$ by \cref{lem:moments}.
We have $\E h_{n,\alpha}^2 \toU 1$ by the combinatorial moment calculation in \cref{lem:combinatorial-variance}.
Therefore, \cref{eq:var1} goes to 0 along $\cU$ which proves the claim.
\end{proof}

\begin{proof}[Proof of \cref{thm:regularity-ultralimit}.]
    We observed above that a spanning set for $\uH_d$ is
    \[
        \Big\{\frac{1}{\sqrt{\alpha!}}\underset{i \in I}{\diamond} [e_{n,i}^{\diamond \alpha_i}] : \al \in \N^I, \gr(\alpha) = d\Big\}\,.
    \]
    By \cref{lem:wick-hermite}, these equal $\{[\psi_\alpha(e_n)] : \alpha \in \N^I, \gr(\alpha) = d\}$.
    The Gaussian convergence in \cref{thm:fourth-moment} implies the $[\psi_\alpha(e_n)]$ are
    orthonormal in $\uH_d$ like the usual Hermite basis.

    For the second part of the theorem,
    if the basis representation of $[p_{n,k}]$ is finite,
    then convergence in distribution follows from \cref{thm:fourth-moment} and the continuous mapping theorem.
    It remains to handle the case where the basis representation of $[p_{n,k}]$ has countably many nonzero terms. The expression for $p_k$ has $\E p_k^2 < \infty$ (hence it is well-defined in $L_2(Z)$) because, by orthonormality in $\grH_{d_k}(Z)$ and $\uH_{d_k}$:
    \begin{equation}\label{eq:ul-1}
        \E p_k^2 = \sum_{\alpha \in \N^I: \gr(\alpha) = d_k} c_{k,\alpha}^2 = \lim_{n \to \cU} \E p_{n,k}^2 < \infty\,.
    \end{equation}
    Fix $\eps > 0$ and a finite set $F_k$ such that $\sum_{\alpha \notin F_k}c_{k,\alpha}^2 \leq \eps$ and define the truncated approximations,
    \[
        [q_{n,k}] = \sum_{\alpha \in F_k} c_{k,\alpha} [\psi_\alpha(e_n)], \qquad q_k = \sum_{\alpha \in F_k} c_{k,\alpha} \psi_\alpha(Z)\,.
    \]
    We have simultaneously:
    \begin{align}\label{eq:eps1}
        \E (p_k - q_k)^2 = \sum_{\alpha \notin F_k}c_{k,\alpha}^2 \leq \eps \qquad \qquad
        \lim_{n \to \cU} \E(p_{n,k} - q_{n,k})^2  = \sum_{\alpha \notin F_k}c_{k,\alpha}^2  \leq \eps
    \end{align}
    Furthermore, since $F_k$ is finite, we have $(q_{n,1}, \dots, q_{n,k}) \overset{d}{\toU} (q_1, \dots, q_k)$ by the finite case.

    To prove $(p_{n,k})_{k \in \N} \overset{d}{\toU} (p_k)_{k \in \N}$ it is equivalent to prove $(p_{n,1},\dots, p_{n,k}) \overset{d}{\toU} (p_1, \dots, p_k)$ for each $k$.
    Fix a bounded 1-Lipschitz test function $\varphi : \R^k\to\R$.
    Let $P_n = (p_{n,1}, \dots, p_{n,k}), P = (p_1, \dots, p_k), Q_n = (q_{n,1}, \dots, q_{n,k}), Q = (q_1, \dots, q_k)$.
    Non-asymptotically,
    \begin{align*}
        & \quad |\E\varphi(P_n) - \E \varphi(P)|\\
        &\leq | \E \varphi(P_n) - \E\varphi(Q_n)| + |\E \varphi(Q_n) - \E\varphi(Q)| + |\E \varphi(Q) - \E\varphi(P)|\\
        &\leq \left(\sum_{i=1}^k \E(p_{n,i} - q_{n,i})^2\right)^{1/2} + |\E \varphi(Q_n) - \E\varphi(Q)| + \left(\sum_{i=1}^k\E (p_i - q_i)^2\right)^{1/2} & (\text{1-Lipschitz})
    \end{align*}
    The right term is at most $\sqrt{k\eps}$ by \cref{eq:eps1}.
    In the limit $n \to \cU$,
    the left term is at most $\sqrt{k\eps}$ and the middle term converges to 0 since $Q_n \overset{d}{\toU} Q$.
    We conclude, for every $\eps > 0$,
    \[
        \lim_{n \to \cU} |\E\varphi(P_n) - \E \varphi(P)| \leq 2\sqrt{k\eps}\,.
    \]

    Take $\eps \downarrow 0$ to get $\lim_{n \to \cU} |\E\varphi(P_n) - \E \varphi(P)| = 0$. Bounded 1-Lipschitz functions determine weak convergence by the Portmanteau theorem, so $P_n \overset{d}{\toU} P$.
\end{proof}

\subsection{Proof of regularity lemmas}

Let $p \in \cH_d(X), \E p^2 \leq 1$ with Hermite coefficient tensor $P$. The idea of the regularity lemma is, if a flattening $P^{(r)}$ has a large singular value $\sigma$, then write
\begin{align*}
    P^{(r)} &= \sig U^{(r)} (V^{(d-r)})^\T + P^{(r)}_{\text{rem}}, \qquad\quad\qquad (U^{(r)} \in \R^{n^r}, V^{(d-r)} \in \R^{n^{d-r}})\\
    p &= \frac{\sig}{c_{r,d-r}} u\diamond v + p_{\text{rem}}, \qquad\qquad \qquad\quad (u \in \cH_{\le r}(X), v \in \cH_{\le d-r}(X))\,.
\end{align*}
$U^{(r)}$ and $V^{(d-r)}$ are the (flattenings of the) Hermite coefficient tensors for lower-degree polynomials $u$ and $v$.
Since $u$ and $v$ are lower degree, they can be decomposed by induction.

On the other hand, removing the singular direction decreases the squared norm,\footnote{The first inequality is not an equality because the Hermite coefficient tensor of $p_{\text{rem}}$ is not $P_{\text{rem}}$ but the symmetrization of $P_{\text{rem}}$. Symmetrization cannot increase the Frobenius norm.}
\[
    \E p_{\text{rem}}^2 \leq \|P_{\text{rem}}^{(r)}\|^2_F = \|P^{(r)}\|_F^2 - \sigma^2 = \|P\|_F^2 - \sigma^2 = \E p^2 - \sigma^2\,.
\]
Since initially $\E p^2 \leq 1$, after removing at most $1/\eps^2$ singular directions, all proper flattenings have spectral norm at most $\eps$.
Once all proper flattenings are small, by \cref{thm:fourth-moment}, the remaining polynomial is Gaussian up to error depending on $\eps$.

The ultrapower is essentially carrying out this approach implicitly, without having to worry about the recursion or the quantitative parameters.
The energy measure $\E p^2 = \|P\|_F^2$ and the recursive decomposition are replaced by orthogonal decomposition in the Hilbert space $\uH_d$.
We now use this approach to prove the regularity lemmas.

\begin{theorem}[Regularity lemma, compactification version]
\label{thm:regularity-lemma-compactness}
    Let $Z$ be a set of graded variables with countably many variables of each grade.
    Let $d_k \in \N$ for $k \in \N$.
    Let $p_{n,k} \in \grH_{d_k}(Z)$ such that $\sup_n \E p_{n,k}^2 < \infty$ for each $k \in \N$.
    Then there exist $p_k \in \grH_{d_k}(Z)$ and a subsequence of $n$ such that $(p_{n,k})_{k \in \N} \xrightarrow{d}(p_k)_{k \in \N}$ along the subsequence.
    Consequently, $\{\Law(p_1, \dots, p_k): p_i \in \grH_{d_i}(Z), \E p_i^2 \leq 1\}$ is weakly compact for each $k$.
\end{theorem}
\begin{proof}
First we show that, without loss of generality, we may assume $p_{n,k}$
are ordinary Gaussian polynomials (all variables have grade one).
By \cref{lem:approximability} applied to $(p_{n,1},\dots, p_{n,n})$,
also using that bounded 1-Lipschitz test functions determine convergence in distribution by the Portmanteau theorem,
there exist ordinary polynomials $q_{n,k}$
such that, for all $k \leq n$,
\begin{align*}
    &\left|\E \varphi(p_{n,1}, \dots, p_{n,k})  - \E\varphi(q_{n,1}, \dots, q_{n,k})\right| \leq \frac{1}{n}
\end{align*}
for all bounded 1-Lipschitz functions $\varphi : \R^k \to \R$.
So, if a subsequence of $q_{n,k}$ converges in distribution, then the same subsequence of $p_{n,k}$ converges in distribution to the same limit.
By \cref{lem:moments}, the approximants $q_{n,i}$ also satisfy $\sup_n \E q_{n,i}^2 < \infty$.
Hence, it suffices to assume $p_{n,k} \in \cH_{d_k}(X_1, \dots, X_n)$ are
Gaussian polynomials such that $\sup_n \E p_{n,k}^2 < \infty$.

Applying \cref{thm:regularity-ultralimit}, there exist
graded Gaussian polynomials $p_k \in \grH_{d_k}(Z)$ such that
\begin{equation}\label{eq:up-1}
    (p_{n,k})_{k \in \N}
    \overset{d}{\toU}(p_k)_{k \in \N}
\end{equation}
Each $p_k$ is a countable or finite sum of Hermite polynomials in the $Z$ variables,
so the entire family $(p_k)_{k\in\N}$ can be represented using at most countably
many graded Gaussian variables.

The weak topology on probability measures on
$\R^\N$ is metrizable (by the sup over bounded 1-Lipschitz functions weighted by $2^{-i}$ in coordinate $i$),
so we may use \cref{lem:ultralimits}(iii) to extract a subsequence
such that \cref{eq:up-1} holds ordinarily.

For the final assertion, apply this result to any sequence
of admissible $(p_{n,1},\dots, p_{n,k})$. The Portmanteau theorem gives, along the extracted subsequence,
\[
    \E p_i^2
    \le \liminf_{n\to\infty}\E p_{n,i}^2
    \le1.
\]
Hence $p_i$ is also admissible, and the set of admissible laws is compact.
\end{proof}

\begin{theorem}[Regularity lemma, approximation version]\label{thm:regularity-approximation}
    For every $k,D \in \N$, $\eps > 0$ there exists $M = M(k,D,\eps) \in \N$ such that the following holds.
    Let $d_1, \dots, d_k \leq D$ (repeats allowed).
    Let $p_i \in \mathcal{H}_{d_i}(X)$ such that $\E p_i^2 \leq 1$ for all $i=1,\dots, k$.
    Then there exist $m \leq M, g_a \leq D$ and $Z_a \in \calH_{g_a}(X)$ for $a \in [m]$, and $q_i \in \grH_{d_i}(Z_1, \dots, Z_m)$ where $Z_a$ has grade $g_a$, such that:
    \begin{enumerate}[(i)]
        \item $\sum_{i = 1}^k \E (p_i - q_i)^2 \leq \eps$
        \item $W_2((Z_1, \dots, Z_{m}), \calN(0,1)^{\otimes m}) \leq \eps$
    \end{enumerate}
\end{theorem}
\begin{proof}
Suppose the assertion fails for some $k,D,\eps$.
Then for every $M$, there is a choice of polynomials which require at least $M$ variables to approximate
within error $4\eps$.
Construct a sequence $p_{n,i} \in \cH_{d_i}(X), \E p_{n,i}^2 \leq 1$ such that the number of Gaussians required to approximate them within error $4\eps$ grows with $n$.
We will use the ultraproduct to approximate
an infinite subsequence using a constant number of variables, yielding the contradiction.

Strictly speaking, the $p_{n,i}$ may not be polynomials since $X$ could be infinite.
For each one, truncate the Hermite expansion to a finite approximant within $\eps$ error.
Then we obtain a sequence of bona fide polynomials $p_{n,i} \in \cH_{d_i}(X_1, \dots, X_n), \E p_{n,i}^2 \leq 1$ such that the number of Gaussians required to approximate them within error $\eps$ grows with $n$.
Embed $[p_{n,i}]$ into $\uH_{d_i}$.

By \cref{thm:regularity-ultralimit}, write the basis representation
\[
    [p_{n,i}]
    =
    \sum_{\gr(\alpha)=d_i}
        c_{i,\alpha}[\psi_\alpha(e_n)].
\]
Choose finite sets $F_i$ of multi-indices of grade $d_i$
such that $\sum_{i=1}^k
    \sum_{\substack{\alpha\notin F_i}}
        c_{i,\alpha}^2
    <\eps/2$.
List the finitely many primitive elements $[e_{n,1}], \dots, [e_{n,m}]$ appearing
in the union of the $F_i$ and for each one select a representative sequence
of polynomials, written $Z_{n,a} \in \cH_{g_a}(X_1, \dots, X_n)$.

Define the graded-degree-$d_i$ polynomial
$q_i(Z_1, \dots, Z_m) = \sum_{\alpha \in F_i} c_{i,\alpha}\psi_\alpha(Z)$.
By asymptotic orthonormality of the $\psi_\alpha(Z)$ in the ultraproduct,
\begin{equation}\label{eq:l2}
    \lim_{n\to\cU}
    \sum_{i=1}^k
       \E\bigl(p_{n,i}
          - q_i(Z_{n,1},\dots,Z_{n,m})\bigr)^2
    =
    \sum_{i=1}^k
    \sum_{\substack{\gr(\alpha)=d_i\\\alpha\notin F_i}}
       c_{i,\alpha}^2
    <\eps/2.
\end{equation}
By \cref{lem:ultralimits}(i), inequality \eqref{eq:l2} holds
with error $\eps$ for all $n$ in a $\cU$-large set. Call this set $A$.

The second part of \cref{thm:regularity-ultralimit} gives $(Z_{n,1}, \dots, Z_{n,m}) \overset{d}{\toU}\calN(0,1)^{\otimes m}$.
We claim that the following set $B$ is also $\cU$-large,
\[
    B =\{n:W_2((Z_{n,1},\dots, Z_{n,m}),\calN(0,1)^{\otimes m})<\eps\}\,.
\]
Otherwise, $\N\setminus B \in\cU$.
By \cref{lem:ultralimits}(iii)
there is a subsequence of $n \in \N\setminus B$
such that $(Z_{n,1}, \dots,Z_{n,m})\xrightarrow{d}\calN(0,1)^{\otimes m}$
along the subsequence.
By \cref{lem:moments}, this also implies convergence in $W_2$,
contradicting the definition of $B$.
Thus $B \in \cU$.

Then $A \cap B$ is also $\cU$-large. Use \cref{lem:ultralimits}(iii) to obtain a  subsequence of $p_{n,i}$ which are approximated up to error $\eps$ by at most $m$ variables,
the desired contradiction.
\end{proof}

\begin{remark}
    The proof of \cref{thm:regularity-approximation} gives a nonconstructive bound on $M(k,D,\eps)$. {GPT-6} Astra asserts that the strategy described at the start of this section can prove an Ackermann-type recursion for $M(k,D,\eps)$, for example, $M(k,2,\eps) \leq Ck^4\eps^{-3}(k+1)^{Ck/\eps}$, a tower-type bound for $D = 3$, and larger bounds for $D>3$. We have not attempted to verify the proof.
\end{remark}

\section{Application: robust Borell inequality}
\label{sec:borell}

\renewcommand{\Stab}{\operatorname{Stab}}

For a parameter $s \in [-1,1]$, the Ornstein--Uhlenbeck noise operator $T_s$ is defined on $L_2(X)$ by
\[ (T_s f)(X) = \E[f(Y) \mid X], \qquad (X,Y) \sim s\text{-correlated Gaussians}\,.\]
Let $\Stab_s(f) = \ip{f, T_s f} = \E_{(X,Y) \sim s\text{-correlated}} f(X)f(Y)$ denote the noise stability of $f$.
For a set $A \subseteq \R^n$ we also write $\Stab_s(A) = \Stab_s(1_A)$.
A \emph{halfspace} is a set of the form $H = \{X : \sum_i c_i X_i \geq c_0\}$
for some coefficients $c_i$ normalized to $\sum_i c_i^2 = 1$.

\begin{theorem}[Robust Borell inequality]
\label{thm:borell-stability}
    For $a \in (0,1)$ let $H_a = (-\infty, \Phi^{-1}(a)) \subseteq \R$.
    For all $a \in (0,1), s \in (0,1), \eps > 0$ there exists $\delta = \delta(a,s, \eps) > 0$ such that:
    if $A \subseteq \R^n, \gam(A) = a$ satisfies $\Stab_s(A) \geq \Stab_s(H_a) - \delta$ then
    there exists a halfspace $H \subseteq \R^n, \gam_n(H) = a$ such that
    $\gam_n(A \symdiff H) \leq \eps$.
\end{theorem}

Mossel and Neeman \cite{mossel2015robust} originally proved
\cref{thm:borell-stability}, and subsequently Eldan improved the bound on $\delta$ \cite{eldan2015two}.
Here we prove the theorem with a nonquantitative bound on $\delta$, with the advantage that the argument is significantly simpler.

The idea of the proof is to show that all maximizers among the space of limits of sets $A \subseteq \R^n$ are halfspaces.
To take the limit of a set $A$, we separately take the limit of each component $\Pi_k 1_A$ using the regularity lemma.
There might be ``mass escaping to infinity'' which disappears in the limit, such as for $A_n = \{X : \prod_{i=1}^n X_i > 0\}$ which has no Hermite weight between degree 0 and $n$.
But, this mass isn't seen by noise stability, so we still get $\Stab_s(A_n) \to \Stab_s(f)$ for the limit $f$.

\begin{lemma}\label{lem:boolean-limit}
    Let $Z$ be a set of graded variables with countably many variables of each grade.
    Let $a \in [0,1]$.
    Let $A_n \subseteq \R^n, \gam_n(A_n) = a$. Then there exists a subsequence of $A_n$ and $f \in L_2(Z)$ such that $0 \leq f \leq 1, \E f = a$, and for all $s \in (0,1)$, along the subsequence,
    \[
        (\Pi_k 1_{A_n})_{k \in \N} \xrightarrow{d} (\Pi_k f)_{k \in \N}, \qquad \lim_{n \to \infty} \Stab_s(A_n) = \Stab_s(f)\,.
    \]
    (Recall that $\Pi_k, T_s,\Stab_s(f)$ are defined on the graded space according to grade, see \eqref{eq:correlation}.)
\end{lemma}
\begin{proof}
    Apply \cref{thm:regularity-lemma-compactness} to $(\Pi_k 1_{A_n})_{k \in \N}$.
    Then there exist graded polynomials $p_k \in \grH_k(Z)$ such that $(\Pi_k 1_{A_n})_{k \in \N} \xrightarrow{d} (p_k)_{k \in \N}$.
    By \cref{lem:moments}, convergence holds for all joint moments. Thus for all $D$,
    \[
        \sum_{k = 0}^D \|p_k\|_2^2 = \lim_{n \to \infty} \sum_{k = 0}^D \|\Pi_k 1_{A_n}\|_2^2 \leq \lim_{n \to \infty}\E 1_{A_n}^2 \leq 1\,.
    \]
    Thus we can define $f = \sum_{k \in \N} p_k$ in $L_2(Z)$ which has the desired components $\Pi_k f = p_k$. The mean of $f$ is, using that non-constant Hermite polynomials are mean-zero,
    \[\E f = \E \Pi_0 f = \lim_{n \to \infty} \E \Pi_0 1_{A_n} = \lim_{n \to \infty} \gam_n(A_n) = a\,.\]
    Next, recall that the noise operator acts as $T_s f = \sum_{k \in \N} s^k \Pi_k f$.
    We have
    \[\abs{\Stab_s(A_n) - \Stab_s(f)} \leq \sum_{k \in \N} s^k \left|\norm{\Pi_k 1_{A_n}}_2^2 - \norm{\Pi_k f}_2^2\right|, \qquad \sum_{k \geq D} s^k \left|\norm{\Pi_k 1_{A_n}}_2^2 - \norm{\Pi_k f}_2^2\right| \leq 2s^{D}\,.\]
    Therefore, for all $\eps > 0$, by taking $D$ large enough then applying convergence in distribution to the components up to $D$, we get $|\Stab_s(A_n) - \Stab_s(f)| \leq \eps$ for all sufficiently large $n$. Hence $\lim_{n \to \infty} \Stab_s(A_n) = \Stab_s(f)$.

    It remains to show $0 \leq f \leq 1$. We use the noise operator,
    proving first $T_s 1_{A_n} \xrightarrow{d} T_s f$.
    We have,
    \[
        \left\|T_s 1_{A_n}  - \sum_{k = 0}^D s^k \Pi_k 1_{A_n} \right\|_2 \leq s^{D+1}, \qquad \left\|T_s f - \sum_{k=0}^D s^k \Pi_k f \right\|_2 \leq s^{D+1}\,.
    \]
    Then for any bounded 1-Lipschitz function $\phi$,
    \begin{align*}
        &\quad |\E \phi(T_s 1_{A_n}) - \E \phi(T_s f)| \\
        &\leq \E |T_s 1_{A_n} - {\textstyle\sum}_{k = 0}^D s^k \Pi_k 1_{A_n}| + |\E \phi({\textstyle\sum}_{k = 0}^D s^k \Pi_k 1_{A_n}) - \E \phi({\textstyle\sum}_{k = 0}^D s^k \Pi_k f)| + \E |{\textstyle\sum}_{k = 0}^D s^k \Pi_k f - T_s f|\\
        &\leq 2s^{D+1} + |\E \phi({\textstyle\sum}_{k = 0}^D s^k \Pi_k 1_{A_n}) - \E \phi({\textstyle\sum}_{k = 0}^D s^k \Pi_k f)|\,.
    \end{align*}
    For any $\eps > 0$, selecting $D$ large enough then applying joint convergence in distribution to the middle term makes the entire expression at most $\eps$.
    Hence $T_s 1_{A_n} \xrightarrow{d} T_s f$.
    
    Take $\phi$ to be a bounded continuous test function which is 0 on $[0,1]$ and positive elsewhere.
    By $T_s 1_{A_n} \xrightarrow{d} T_s f$ and the observation $T_s 1_A \in [0,1]$ for any set $A$,
    \[
        \E \phi(T_s f) = \lim_{n \to \infty} \E \phi(T_s 1_{A_n}) = 0\,.
    \]
    Since $\phi$ is nonnegative, we must have $T_s f \in [0,1]$ almost surely.
    Finally, since $T_s f$ converges to $f$ in $L_2$ as $s \uparrow 1$, then $f \in [0,1]$ almost surely.
\end{proof}

\begin{lemma}[Characterization of equality case]
\label{lem:characterize-equality}
    Let $a \in (0,1), s \in (0,1)$.
    Let $Z$ be a set of graded variables with countably many variables of each grade
    and let $f \in L_2(Z)$ such that $0 \leq f \leq 1$, $\E f = a$, and $\Stab_s(f) = \Stab_s(H_a)$.
    Up to a set of measure zero, then $f = 1_H$ for some $H = \{X : \sum_i c_i X_i \geq c_0\}$ with $\sum_i c_i^2 = 1$ supported on grade-1 variables.
\end{lemma}
\begin{proof}
    The noise stability is $\Stab_s(f) = \sum_{k \in \N} s^k \norm{\Pi_k f}_2^2$.
    Graded variables in $f$ increase the graded-degree, and hence strictly worsen the noise stability: $\sum_{k \in \N}s^k \|\Pi_k f\|_2^2 \leq \sum_{k \in \N} s^k \|\tilde \Pi_k f\|_2^2$ where $\tilde \Pi_k$
    treats all variables as grade 1.
    Therefore maximizers can use only grade 1 variables.

    For $0 \leq f \leq 1$, write $f = \E_{\mu} g$ for a mixture $\mu$ of 0/1-valued $g$.
    Since $\Stab_s(f)$ is strictly convex, maximizers must be 0/1-valued functions (i.e.\ sets).
    Strict convexity of $f \mapsto \Stab_s(f)$ follows from the observation that $\Stab_s(f) = \ip{f, T_s f}$ is a positive definite quadratic form for $s \in (0,1]$.
    
    It remains to characterize the maximizers among $A \subseteq \R^X$ with a countable number of variables $X$.
    For a finite set of variables $X$ and
    Mossel and Neeman \cite[Theorem 1.3]{mossel2015robust} characterize the maximizers
    as halfspaces $H = \{X : \sum_i c_i X_i \geq c_0\}$. The case of countably many variables follows by notational changes.\footnote{Except the line ``$\nabla v_t(X) = \nabla w_t(Y)$ almost surely implies $\nabla v_t = \nabla w_t$ are almost surely equal to some constant, since the conditional distribution of $Y$ given $X$ is fully supported'' ($X, Y$ are a pair of $\rho$-correlated Gaussians).
    Change this line to: We deduce $\nabla v_t(X) = \nabla w_t(Y) = \nabla v_t(X')$ almost surely where $X, X'$ are a pair of $\rho^2$-correlated Gaussians. It holds generally that, $\E \|\nabla v_t(X) - \nabla v_t(X')\|_2^2 \geq 2(1-\rho^2)\E\|\nabla v_t(X) - \E \nabla v_t(X)\|_2^2$. This implies $\nabla v_t(X) = \E \nabla v_t(X)$ almost surely. Hence $\nabla v_t(X)$ is constant, and $\nabla w_t(Y) = \nabla v_t(X)$ is the same constant.}
\end{proof}

For a set $A \subseteq \R^n$, let $m_A$ denote its barycenter,
\[
    m_A := \int_{\R^n} x 1_A(x) d\gam_n(x), \qquad (\Pi_1 1_A)(x) = \ip{m_A, x}\,.
\]
The candidate halfspace approximation for $A$ is $H = \{x \in \R^n: \ip{\frac{m_A}{\norm{m_A}_2}, x} \geq \Phi^{-1}(1 - \gam_n(A))\}$.
The following lemma establishes an upper bound on the approximation error.
\begin{lemma}\label{lem:halfspace-barycenter}
    Let $a \in (0,1)$ and $A \subseteq \R^n, \gam_n(A) = a$.
    Set $H = \{x \in \R^n : \ip{x, \frac{m_A}{\norm{m_A}_2}} \geq \Phi^{-1}(1-a)\}$. Then
    \[
        \gam_n(A \symdiff H)^2 \leq \frac{4\sqrt{2}}{\sqrt{\pi}}\left(\norm{m_H}_2 - \norm{m_A}_2\right)\,.
    \]
\end{lemma}
\begin{proof}
    Let $u = \frac{m_A}{\norm{m_A}_2}$ and $t = \Phi^{-1}(1-a)$ so $H = \{x : \ip{u,x} \geq t\}$. The unit vector $u$ points in the direction of both $m_A$ and $m_H$, so
    \begin{align*}
        \int_{\R^n} \ip{u,x} 1_A(x) d\gam_n(x) = \ip{u,m_A} = \norm{m_A}_2, \qquad \int_{\R^n}\ip{u,x} 1_H(x) d\gam_n(x) = \ip{u, m_H} = \norm{m_H}_2.
    \end{align*}
    Therefore, using that $\int_{\R^n} (1_A(x) - 1_H(x)) d\gam_n(x) = \gam_n(A) - \gam_n(H) = 0$,
    \begin{align}\label{eq:b1}
        \norm{m_H}_2 - \norm{m_A}_2 = \int_{\R^n} \left(\ip{u,x} -t \right)(1_H(x) - 1_A(x)) d\gam_n(x) = \int_{A \symdiff H} |\ip{u,x} - t| d\gam_n(x)\,.
    \end{align}
    In the rightmost expression,
    the worst case for $A \symdiff H$ is if it is a slab around the plane $\ip{u,x} = t$.
    Therefore, we reduce to a one-dimensional integral,
    \begin{equation}\label{eq:b2}
    \int_{A \symdiff H} |\ip{u,x} - t| d\gam_n(x) \geq \inf_I \int_I |x - t| d\gam(x)
    \end{equation}
    where the one-dimensional integral is Gaussian since $\ip{u,x}$ is standard Gaussian,
    and $I \ni t$ is an interval with $\gam(I) = \gam_n(A \symdiff H) =: \delta$.
    Since the Gaussian PDF is pointwise upper bounded by $\frac{1}{\sqrt{2\pi}}$,
    $I$ must contain at least $\frac\delta2$ mass outside the interval $[t - \frac{\sqrt{\pi}}{2\sqrt{2}}\delta, t + \frac{\sqrt{\pi}}{2\sqrt{2}}\delta]$.
    All such points have $|x-t| \geq \frac{\sqrt{\pi}}{2\sqrt{2}}\delta $. Therefore,
    \begin{align}\label{eq:b3}
        \int_I |x - t| d\gam(x) \geq \frac{\sqrt{\pi}}{2\sqrt{2}}\delta \cdot \frac{\delta}{2} =\frac{\sqrt{\pi}}{4\sqrt{2}} \delta^2\,.
    \end{align}
    Combining \eqref{eq:b1}, \eqref{eq:b2}, \eqref{eq:b3} proves the claim.
\end{proof}

\begin{proof}[Proof of \cref{thm:borell-stability}]
    Assume for the sake of contradiction that the assertion fails for some $a,s,\eps$.
    Then there is a sequence of sets $A_n, \gam(A_n) = a$
    such that $\Stab_s(H_a) - o(1) \leq \Stab_s(A_n) \leq \Stab_s(H_a)$ but $\gam_n(A_n\symdiff H) > \eps$
    for all halfspaces $H$ of measure $a$.
    By \cref{lem:boolean-limit},
    there exists $f \in L_2(Z)$ such that, along a subsequence,
    \[
        (\Pi_k 1_{A_n})_{k \in \N} \xrightarrow{d} (\Pi_k f)_{k \in \N}, \qquad 0 \leq f \leq 1, \qquad \E f = a, \qquad \Stab_s(f) = \Stab_s(H_a)\,.
    \]
    By \cref{lem:characterize-equality},
    $f$ is a halfspace: $f = 1_H$ for some $H = \{X : \sum_i c_i X_i \geq c_0\}$.
    By the joint convergence in \cref{lem:boolean-limit},
    \[
        \lim_{n \to \infty} \|\Pi_1 1_{A_n}\|_2^2 = \|\Pi_1 1_H\|_2^2\,.
    \]
    In the notation of \cref{lem:halfspace-barycenter}, $\|\Pi_1 1_{A_n}\|_2^2 = \|m_{A_n}\|_2^2$ and $\|\Pi_1 1_H\|_2^2 = \|m_H\|_2^2$.
    Therefore, by \cref{lem:halfspace-barycenter}, there are halfspaces $H_n$ such that $\gam_n(A_n \symdiff H_n) \to 0$, contradicting the assumption.
\end{proof}

\bibliographystyle{alpha}
{\footnotesize\bibliography{references}}

\appendix

\section{Probability preliminaries}
\label{sec:hypercontractivity}

\begin{lemma}[{\cite[Example 2.4]{billingsley1999convergence}}]
Let $x_{n,k}$ be real-valued random variables. Endow $\R^\N$ with the product topology. The following are equivalent:
\begin{enumerate}[(i)]
    \item $(x_{n,k})_{k \in \N} \xrightarrow{d} (x_k)_{k \in \N}$
    \item $(x_{n,1},\dots, x_{n,k}) \xrightarrow{d} (x_1, \dots, x_k)$ for all $k$
    \item $(x_{n,j})_{j \in J} \xrightarrow{d} (x_j)_{j \in J}$ for all finite sets $J \subseteq \N$.
\end{enumerate}
\end{lemma}

\begin{lemma}[Method of moments {\cite[Theorems 29.4, 30.1, 30.2]{billingsley1995probability}}]\label{lem:moment-method} 
Let $x_{n,k}$ be real-valued random variables. Assume that $(x_1, \dots, x_k)$ have a jointly Gaussian distribution. Assume for all finitely supported $\alpha \in \N^k$,
\[
    \E \prod_{i=1}^k x_{n,i}^{\alpha_i} \longrightarrow \E \prod_{i =1}^kx_i^{\alpha_i}\,.
\]
Then $(x_{n,1},\dots, x_{n,k}) \xrightarrow{d} (x_1,\dots, x_k)$.
\end{lemma}

\begin{lemma}[Gaussian hypercontractive inequality]
\label{lem:hypercontractivity}
For all $p\in\cH_{\le d}(X)$ and $q\ge2$,
\begin{equation*}
 \norm p_q\le(q-1)^{d/2}\norm p_2\,.
\end{equation*}
\end{lemma}

\begin{lemma}
\label{lem:moments}
Let $d_1,\dots,d_k\in\N$ and $p_{n,i} \in \cH_{\le d_i}(X)$.
If $(p_{n,1}, \dots, p_{n,k}) \xrightarrow{d} (p_1,\dots, p_k)$ then all mixed moments converge, and the convergence also holds in $W_p$ for $p \in [1,\infty)$.
The same holds for graded polynomials $p_{n,i} \in \grH_{d_i}(Z)$.
\end{lemma}
\begin{proof}
Write $P_n=(p_{n,1},\dots,p_{n,k})$ and
$P=(p_1,\dots,p_k)$.
Since $P_n\xrightarrow{d}P$, each coordinate sequence
$(p_{n,i})_{n\in \N}$ has tight laws.
Recall a collection of probability measures on $\R^I$ is \emph{tight} if for every $\eps > 0$
there is compact $K \subseteq \R^I$ such that $\mu(K) \geq 1-\eps$ for every measure $\mu$ in the collection.

By tightness, the Paley--Zygmund inequality, and \cref{lem:hypercontractivity},
we have for $p_{n,k} \neq 0$,
\[
    \Pr\left(
        |p_{n,k}|\ge\frac{\norm{p_{n,k}}_2}{\sqrt2}
    \right)
    \ge
    \frac{\norm{p_{n,k}}_2^4}
         {4\norm{p_{n,k}}_4^4}
    \ge \frac{1}{4\cdot9^{d_k}}.
\]
Thus tightness implies $\sup_n\norm{p_{n,k}}_2<\infty$.
Applying hypercontractivity again, for every $q\ge2$,
\[
    \sup_n\norm{p_{n,k}}_q
    \le (q-1)^{d_k/2}\sup_n\norm{p_{n,k}}_2
    <\infty.
\]
Consequently, we deduce $\sup_n \E\norm{P_n}^q<\infty$ and, by the Portmanteau theorem, $\E\norm{P}^q< \infty$ for all $q \geq 1$.
Then for every continuous function $\varphi$ of
polynomial growth, the sequence $(\varphi(P_n))_{n\in\N}$ is uniformly
integrable. The continuous mapping theorem and
\cite[Theorem~3.5]{billingsley1999convergence} imply
\[
    \E\varphi(P_n)\longrightarrow\E\varphi(P).
\]
Thus we have convergence of all mixed moments.
Taking $\varphi(x)=\|x\|^q$
gives convergence of the $q$th absolute moments.
Together with weak convergence, this implies
$W_q(\Law(P_n),\Law(P))\to0$ by
\cite[Theorem~6.9]{villani2009optimal}.
Forgetting grades shows $\grH_{d}(Z) \subseteq \cH_{\le d}(Z)$ so the same argument applies to graded polynomials.
\end{proof}

\section{Combinatorial proof of the fourth moment theorem}
\label{sec:fourth-moment}

We use the tensor notation from \cref{sec:tensors}.
We adopt the Nualart and Peccati~\cite{nualart2005central} notation that $P \otimes_r P$ equals $(P^{(r)})^\T P^{(r)}$ unflattened into a tensor.
$P \otimes_r P$ is a non-symmetric tensor with order $2(d-r)$ when $P$ has order $d$.

\begin{theorem}[{Fourth moment theorem \cite{nualart2005central,peccati2004gaussian}}]
\label{thm:fourth-moment}
Let $d_1, \dots, d_k \in \N \setminus \{0\}$ (repeats allowed). Let $p_{n,i} \in \cH_{d_i}(X_1, \dots, X_n)$.
    Let $P_{n,i}$ be their Hermite coefficient tensors.
    Assume that $\lim_{n \to \infty}\E p_{n,i}p_{n,j} = \Sig_{ij}$ for some matrix $\Sig \in \R^{k \times k}$.
    The following are equivalent:
\begin{enumerate}[(i)]
    \item $(p_{n,1}, \dots, p_{n,k}) \xrightarrow{d} \calN(0, \Sig)$
    \item $\E p_{n,i}^4 \rightarrow 3\Sig_{ii}^2$ for all $i = 1,\dots, k$.
    \item $\|P_{n,i} \otimes_r P_{n,i}\|_F \rightarrow 0$ for all $0 < r < d_i$ and $i=1, \dots, k$.
    \item $\norm{P_{n,i}^{(r)}} \rightarrow 0$ for all $0 < r < d_i$ and $i = 1, \dots, k$.
\end{enumerate}
\end{theorem}
Equivalence of (i),(ii),(iii) was proven by \cite{nualart2005central,nualart2008central}.
It is easy to check that (iii) is equivalent to (iv)
using the relations:
\begin{equation}\label{eq:norm-compare}
    \|P^{(r)}\|^2 \leq \|P \otimes_r P\|_F \leq \|P^{(r)}\| \|P\|_F\,.
\end{equation}
To see these relations, let $s_j$ be the singular values of $P^{(r)}$. \cref{eq:norm-compare} follows from the equations,
\[
    \|P^{(r)}\| = \max_j s_j, \qquad \|P \otimes_r P\|_F^2 = \sum_{j} s_j^4, \qquad \|P\|_F^2 = \sum_j s_j^2\,.
\]
Since $\|P_{n,i}\|_F^2 = \E p_{n,i}^2$ converges to the constant $\Sig_{ii}$, then \cref{eq:norm-compare} implies $\|P \otimes_r P\|_F$ vanishes if and only if $\|P_{n,i}^{(r)}\|$ vanishes. Hence (iii) is equivalent to (iv).

We now give a direct proof of \cref{thm:fourth-moment} using the moment method (see also \cite{nourdin2011yet} for a prior combinatorial proof).
The moments of Hermite polynomials have natural combinatorial interpretations in terms of matchings.
\begin{lemma}[{\cite[Chapter 3]{janson1997gaussian}}]
\label{fact:hermite-matchings}
    Let $\al^1, \dots, \al^\ell \in \N^I$ be finitely supported. Then:
    \begin{align*}
        \E\left[\prod_{j=1}^\ell \He_{\al^j}(X)\right] &= \sum_{M \in \calM(\al^1,\dots,\al^\ell)} \prod_{uv \in M} \mathbf{1}_{u=v}
    \end{align*}
     where we take $\ell$ groups of objects, with the $j$th group being a multiset of size $|\al^j|$ with $\al^j_i$ copies of each $i \in I$. Let $\calM(\al^1,\dots, \al^\ell)$ be the set of perfect matchings on the disjoint union of the groups such that no two objects are matched within a group. The equality $u=v$ is on $I$.
\end{lemma}

A \emph{tensor network} is a multigraph $\pi$ along with a collection of symmetric tensors $T_a \in (\R^{n})^{\otimes \deg(a)}_\sym$ associated to the vertices of $\pi$.
The value of the tensor network is:
\[
    \val_\pi(T) = \sum_{\varphi : E(\pi) \to [n]} \prod_{a \in V(\pi)} (T_a)_{\varphi(a)}
\]
where $\varphi(a)$ is the multiset of labels on the edges incident to $a$ (the order of the incident edges does not matter since $T_a$ is assumed to be symmetric).

    \begin{lemma}\label{lem:tensor-frob}
        Let $(T, \pi)$ be a tensor network such that $\pi$ has no self-loops.
        Then:
        \[
            |\val_\pi(T)| \leq \prod_{a \in V(\pi)} \norm{T_a}_F.
        \]
    \end{lemma}
    \begin{proof}
        $\val_\pi(T)$ factors across connected components of $\pi$, so we may assume $\pi$ is connected.
        Use induction on $|V(\pi)|$. We also upgrade the statement to allow non-symmetric tensors (in this case, each vertex of $\pi$ comes with an ordering of the incident edges).
        
        In the base case, $\pi$ is a single multiedge between two tensors $R,S$,
        \[
            \abs{\val_\pi(T)} = \left|\sum_{i \in [n]^d} R_i S_i\right| \leq \left(\sum_{i \in [n]^d} R_i^2\right)^{1/2} \left(\sum_{i \in [n]^d} S_i^2\right)^{1/2} = \norm{R}_F \norm{S}_F\,.
        \]
        
        In the induction step, let $a,b$ be the endpoints of a multiedge in $\pi$ with associated tensors $R,S$.
        Let $\pi'$ be $\pi$ with the multiedge contracted.
        Let $Q$ be the contraction of $R,S$ along the multiedges between $a,b$
        and associate $Q$ to the contracted vertex in $\pi'$.
        Call this new tensor network $(T', \pi')$. Then $\val_\pi(T) = \val_{\pi'}(T')$ by definition. By induction,
        \[
            \abs{\val_{\pi'}(T')} \leq \norm{Q}_F \prod_{c \in V(\pi) \setminus \{a,b\}} \norm{T_{c}}_F\,.
        \]
        We can express $Q$ as a flattened matrix multiplication $R^{(L_a,R_a)} S^{(L_b,R_b)}$
        followed by unflattening, where $R_a,L_b$ are the edges between $a,b$.
        Using $\norm{AB}_F \leq \norm{A}_F \norm{B}_F$ on the matrix multiplication,
        \[
            \norm{Q}_F \prod_{c \in V(\pi) \setminus \{a,b\}} \norm{T_c}_F \leq \norm{R}_F \norm{S}_F \prod_{c \in V(\pi) \setminus \{a,b\}} \norm{T_{c}}_F
            = \prod_{a \in V(\pi)} \norm{T_a}_F\,. \qedhere
        \]
    \end{proof}

\begin{lemma}\label{lem:moment}
    Let $d_1, \dots, d_k \geq 1$. Let $p_a \in \cH_{d_a}(X_1, \dots, X_n)$ such that $\E p_a^2 \leq 1$ and let $P_a$ be their Hermite coefficient tensors.
    Assume that $\norm{P_a^{(r)}} \leq \delta$ for all $a \in [k]$ and $0 < r < d_a$.
    Let $Z=(Z_1,\ldots,Z_k)$ be a centered Gaussian vector with the same
    covariance matrix as $(p_1,\ldots,p_k)$, and let $D=\sum_{a=1}^k d_a$.
    If $D$ is even, then
      \begin{equation}\label{eq:mixed-moment-comparison}
        \left|\mathbb E\prod_{a=1}^k p_a
              -\mathbb E\prod_{a=1}^k Z_a\right|
        \le
        \frac{(D-1)!!}{\sqrt{d_1!\cdots d_k!}}\,\delta.
      \end{equation}
  If $D$ is odd, both expectations vanish.
\end{lemma}
\begin{proof}
    Take a set of $k$ vertices, where $a$ is incident to $d_a$ labeled half-edge endpoints. Let $\mathcal D$ be the
    set of perfect matchings of all $D$ half-edges such that no two half-edges at the
    same vertex are matched. $\pi\in\mathcal D$ forms a tensor network with $P_1, \dots, P_k$.
    We have,
    \begin{align*}
      \mathbb E\prod_{a=1}^k p_a
      &=\frac{1}{\sqrt{d_1!\cdots d_k!}} \sum_{i_a \in [n]^{d_a}} \prod_{a=1}^k (P_a)_{i_a} \E\left[\prod_{a=1}^k\He_{i_a}(X)\right] & (\text{\cref{eq:coefficient-tensor}})\\
      &=\frac{1}{\sqrt{d_1!\cdots d_k!}}
        \sum_{\pi\in\mathcal D}
           \operatorname{val}_\pi(P_1,\ldots,P_k). & (\text{\cref{fact:hermite-matchings}}) \numberthis \label{eq:wick-diagram-expansion}
    \end{align*}

\cref{eq:wick-diagram-expansion} is a ``Feynman diagram expansion'' for the moment $\E \prod_{a=1}^k p_a$.
When $D$ is odd, $\mathcal D$ is empty, and the expectation is zero as needed.

Call a diagram $\pi$ a \emph{pair diagram} if every connected component
has exactly two vertices. Such a component on vertices $a,b$ is
possible only if $d_a=d_b=d$. Its $d!$ possible matchings all
give $\langle P_a,P_b\rangle_F = \E [p_a p_b]$, and its normalization in
\eqref{eq:wick-diagram-expansion} is $1/d!$. Thus the sum over pair
diagrams equals
\[
  \sum_{\text{pairings }\eta\text{ of }\{1,\ldots,k\}}
    \prod_{\{a,b\}\in\eta}\mathbb E[p_ap_b]
  =\mathbb E\prod_{a=1}^k Z_a.
\]

Every remaining diagram has a connected component with at least
three vertices.
Choose adjacent vertices $a,b$ in this component,
and let $r\ge1$ be the number of edges between them. These edges
cannot be all edges incident to both vertices, so, after interchanging $a,b$ if
necessary, $0<r<d_a$.
As in the proof of \cref{lem:tensor-frob}, form a new tensor network $\pi'$ by contracting the vertices $a,b$ and putting at the contracted vertex the tensor $Q$ which is the contraction of $P_a, P_b$. This defines $\val_\pi(P) = \val_{\pi'}(P, Q)$.
By \cref{lem:tensor-frob},
\[
    \abs{\val_{\pi'}(P, Q)} \leq \|Q\|_F \prod_{c \in [k] \setminus \{a,b\}} \|P_c\|_F\,.
\]
We can express $Q$ as the unflattening of a matrix multiplication $(P_a^{(r)})^\T P_b^{(r)}$. Use $\|AB\|_F \leq \|A\|\|B\|_F$ on the matrix multiplication to deduce,
\[
    \|Q\|_F \prod_{c \in [k] \setminus \{a,b\}} \|P_c\|_F \leq \|P_a^{(r)}\| \prod_{b \in [k] \setminus \{a\}} \|P_b\|_F \leq \|P_a^{(r)}\| \leq \delta\,.
\]
The second-to-last inequality uses $\|P_b\|_F^2 = \E p_b^2 \leq 1$.
Thus $\abs{\val_\pi(P)} \leq \delta$ for each non-pair diagram in \cref{eq:wick-diagram-expansion}.
There are at most $(D-1)!!$ matchings of $D$ slots, which proves
\cref{eq:mixed-moment-comparison}.
\end{proof}

\begin{proof}[Proof of \cref{thm:fourth-moment}.]
(i) $\Longrightarrow$ (ii). Convergence in distribution implies convergence of moments by \cref{lem:moments}.

\noindent (ii) $\Longrightarrow$ (iii). Let $p \in \cH_d(X_1,\dots, X_n)$ have Hermite coefficient tensor $P$.
The fourth cumulant $\kappa_4 := \E p^4 - 3(\E p^2)^2$ can be written as a sum-of-squares expression in terms of tensor contractions as \cite[Proof of Proposition 3, (i) $\Longrightarrow$ (ii)]{nualart2005central},
\begin{equation}\label{eq:kappa4}
    \kappa_4 = \sum_{r=1}^{d-1} \binom{d}{r}^2 \left(\|P \otimes_r P\|_F^2 + \binom{2d-2r}{d-r} \|\Sym(P \otimes_r P)\|_F^2\right)\,.
\end{equation}
Therefore, if $\kappa_4 \rightarrow 0$ as $n \to \infty$ then all of the tensor norms $\|P \otimes_r P\|_F$ must also vanish.

Here is a combinatorial interpretation of \eqref{eq:kappa4}.
Following \cref{eq:wick-diagram-expansion},
\[
    \E p^4 = \frac{1}{(d!)^2}\sum_{\pi \in \calD} \val_\pi(P, P, P, P)
\]
where $\calD$ is the set of perfect matchings on four vertices with $d$ incident half-edges such that no half-edges are matched at the same vertex.

Suppose $\pi$ has $r$ edges matched between vertex 1 and 2.
Then there are also $r$ edges matched between vertex 3 and 4 and there are $\left(\binom{d}{r}^2r!\right)^2$ ways to choose this part of $\pi$.
We interpret the edges between 1,2 as contracted in $P \otimes_r P$
and the edges outgoing from 1,2 as the indices of $P \otimes_r P$.

There are $(2d-2r)!$ ways to choose the remaining part of $\pi$.
If the remaining part of $\pi$ matches the half-edges at 1 with the half-edges at 3 (i.e.\ $\pi$ is a cycle graph),
we obtain exactly $\val_\pi(P,P,P,P) = \|P \otimes_r P\|_F^2$.
However, in general, the sum over $\pi$ treats all outgoing edges from $\{1,2\}$
in the same way,
so the average value of $\val_\pi(P,P,P,P)$ is $\|\Sym(P \otimes_r P)\|_F^2$.
In total,
\begin{align}
    \frac{1}{(d!)^2}\sum_{\pi \in \calD} \val_\pi(P, P, P, P) &= \sum_{r=0}^d\frac{(2d-2r)!}{(d!)^2}\left(\binom{d}{r}^2r!\right)^2 \|\Sym(P \otimes_r P)\|_F^2 \notag\\
    &= \sum_{r = 0}^d \binom{d}{r}^2 \binom{2d-2r}{d-r} \|\Sym(P \otimes_r P)\|_F^2
\end{align}

We re-analyze the terms $r=0$ and $r=d$ separately.
The term $r=d$ equals $\|P\|_F^4 = (\E p^2)^2$.
All $\pi$ in the term $r = 0$ are cyclic graphs (or disconnected).
Let $s$ be the number of 1-3 edges which equals the number of 2-4 edges.
There are $\binom{d}{s}^4 (s!)^2((d-s)!)^2 = \binom{d}{s}^2 (d!)^2$ choices for $\pi$.
For these cases, $\val_\pi(P,P,P,P) = \|P \otimes_s P\|_F^2$.
Therefore,
\begin{align*}
    \frac{1}{(d!)^2}\sum_{\pi \in \calD} \val_\pi(P, P, P, P) &=  \|P\|_F^4 + \sum_{s = 0}^d \binom{d}{s}^2\|P \otimes_s P\|_F^2 + \sum_{r = 1}^{d-1}\binom{d}{r}^2 \binom{2d-2r}{d-r} \|\Sym(P \otimes_r P)\|_F^2\\
    &= 3\|P\|_F^4 + \sum_{s = 1}^{d-1} \binom{d}{s}^2 \|P \otimes_s P\|_F^2 + \sum_{r = 1}^{d-1}\binom{d}{r}^2 \binom{2d-2r}{d-r} \|\Sym(P \otimes_r P)\|_F^2
\end{align*}
Recall this equals $\E p^4$. Subtract $3\|P\|_F^4 = 3(\E p^2)^2$ to obtain $\kappa_4$.

\noindent (iii) $\Longrightarrow$ (iv). Proven immediately after the theorem statement.

\noindent 
(iv) $\Longrightarrow$ (i). Let $Z_n$ be a centered Gaussian
    vector with the same covariance matrix as
    $(p_{n,1},\ldots,p_{n,k})$.
    The coordinatewise variances of $Z_n$ are uniformly bounded
    since $\E p_{n,i}^2 \to \Sig_{ii}$,
    say by a constant $C \geq 1$.
    We may apply \cref{lem:moment}
    to get, for every $(m_1,\ldots,m_k)\in\N^k$, letting $|m| = \sum_a m_a$ and $D = \sum_a m_ad_a$ and $\delta_n = \max_{a \in [k]} \max_{0 < r < d_a} \|P_{n,a}^{(r)}\|$,
    \[
      \abs{\mathbb E\prod_{a=1}^k p_{n,a}^{m_a}
      -\mathbb E\prod_{a=1}^k Z_{n,a}^{m_a}} \leq C^{|m|}(D-1)!!\,\delta_n
      \longrightarrow0.
    \]
    The covariances and hence also all moments of $Z_{n}$ converge to those of $\calN(0,\Sig)$.
    So by the method of moments in \cref{lem:moment-method}, $(p_{n,1},\dots, p_{n,k}) \xrightarrow{d} \calN(0, \Sig)$.
\end{proof}

\subsection{Omitted moment calculations}

\begin{lemma}\label{lem:omitted-moment}
    For $i,r \in \N$, let $S_{ir} = \frac{1}{\sqrt{N}} \sum_{j=1}^N \psi_r(G_{ir,j})$. For finitely-supported $\al \in \N^{\N \times \N}$, let 
    \[
        U_{m,ir} = \frac{1}{N^{m/2}} \sum_{\substack{j_1,\dots, j_m \in [N]\\j_1,\dots, j_m\text{ distinct}}} \psi_r(G_{ir,j_1})\psi_r(G_{ir,j_2})\cdots \psi_r(G_{ir,j_m}), \qquad U_\alpha = \prod_{i,r} U_{\alpha_{ir}, ir}\,.
    \]
    Then for $G_{ir,j}\overset{\text{i.i.d.}}{\sim} \calN(0,1)$,
    \[
        \E (U_\al - \He_\al(S))^2 \leq O_\alpha(N^{-1})\,.
    \]
\end{lemma}
\begin{proof}
Consider generic $S = N^{-1/2} \sum_{j=1}^N \psi_r(G_j)$ and $U_m = N^{-m/2}\sum_{j_1,\dots, j_m \text{ distinct}} \prod_{a=1}^m\psi_r(G_{j_a})$.
Using \cref{fact:hermite-matchings},
\[
    \E S^{2\ell} = \frac{1}{(Nr!)^{\ell}}\sum_{M \in \calM(2\ell,r)} \sum_{\varphi: [2\ell] \to [N]} \prod_{uv \in M}\mathbf{1}_{\varphi(\text{group}(u)) = \varphi(\text{group}(v))}\,.
\]
where we have $2\ell$ groups of $r$ vertices each, $\calM(2\ell, r)$ is the set of perfect matchings on the vertices such that no two vertices are matched within a group,
and the indicator function checks whether all matched pairs use the same Gaussian $G_{k}$ on their group. The inner summation is $N^{c(M)}$ where $c(M)$ is the number
of connected components of $M$ when the groups are contracted. Hence,
\begin{equation}\label{eq:s1}
    \E S^{2\ell} = \frac{1}{(Nr!)^{\ell}}\sum_{M \in \calM(2\ell,r)} N^{c(M)} = (2\ell-1)!! + O_\alpha(N^{-1})\,.
\end{equation}
The second equality separates the dominant terms where $\calM$ is a perfect matching between groups from the remaining $O_\alpha(1)$ terms.

Express the Hermite polynomials in the monomial basis and use \cref{eq:s1} to deduce $\E \He_m(S)^2 = \E \He_m(X)^2 + O_\alpha(N^{-1}) = m! + O_\alpha(N^{-1})$ for standard Gaussian variables $X$. We also have
\begin{equation}\label{eq:s2}
    \E \He_m(S) U_m  = \E U_m^2= m! \frac{(N)_m}{N^m} = m! + O_m(N^{-1})\,.
\end{equation}
where $(N)_m$ is the falling factorial. Indeed, the first equality observes that the only component of $\He_m(S)$ which is not orthogonal to $U_m$ is its ``leading monomial'' $U_m$ itself.
The second equality observes that nonzero terms in $\E U_m^2$ are perfect matchings between the indices in $U_m$.
Therefore,
\[
    \E (\He_m(S) - U_m)^2 = \E \He_m(S)^2 - 2\E\He_m(S)U_m + \E U_m^2 = O_{\alpha}(N^{-1})\,.
\]
Independence across the $i,r$ coordinates implies the same holds for the multivariate case.
\end{proof}

\begin{lemma}\label{lem:combinatorial-variance}
    Let $d_1, \dots, d_k \geq 1$ and $D = \sum_{a=1}^k d_a$. Let $p_a \in \cH_{d_a}(X_1, \dots, X_n)$ such that $\E p_a^2 = 1$ and let $P_a$ be their Hermite coefficient tensors.
    Assume that $\|P_a^{(r)}\| \leq \delta$ for all $a \in [k]$ and $0 < r < d_a$.
    Then
    \[
        \left|\E\left(\overset{k}{\underset{a = 1}{\diamond}} p_a\right)^2 - \sum_{M \in \calM(2,k)} \prod_{uv \in M} \E[p_u p_v]\right| \leq \frac{D!}{d_1!\cdots d_k!}\delta
    \]
    where $\calM(2,k)$ is the set of perfect matchings between two copies of $[k]$.
\end{lemma}
\begin{proof}[Proof sketch.]
    The tensor $P = \sqrt{\frac{D!}{d_1!\cdots d_k!}}\otimes_{a=1}^k P_a$ represents the top product $\diamond_{a=1}^k p_a$.
    Then we write a diagrammatic expansion for $\E \left(\overset{k}{\underset{a = 1}{\diamond}} p_a\right)^2$ in the same way as \cref{lem:moment}.
    Separating out the dominant and subdominant perfect matchings
    in the same way proves the claim.
\end{proof}

\end{document}